\documentclass{article}[10pt]       
\usepackage{amsmath,amssymb,amsthm,txfonts}

\usepackage{tikz, framed}
\usetikzlibrary{arrows,shapes,matrix,decorations.pathmorphing}
\usetikzlibrary{backgrounds}
\usetikzlibrary{decorations.pathreplacing,positioning}

 \usepackage{graphicx}
  
 \newcommand{\NN}{\mathbb{N}}

  \newcommand{\cT}{{\sf{T}}}
 \newcommand{\cN}{{\sf{N}}}
  \newcommand{\cB}{{\sf{B}}}

\newcommand{\bM}{{\overline{M}}}
 
  \newcommand{\cG}{{\sf{G}}}

  \newcommand{\ck}{{\mathbf{k}}}
  \newcommand{\st}{{\textrm{such that\ }}}

 \newcommand{\kT}{\mathcal{T}}
\newcommand{\kN}{\mathcal{N}}

\newcommand{\kA}{\mathcal{A}}
\newcommand{\kM}{\mathcal{M}}
\newcommand{\kF}{\mathcal{F}}

\newcommand{\Nf}{\textrm{Nf}}

\def\Bbb#1{{\mathbb #1}}

\theoremstyle{plain} 

\newtheorem{theorem}{Theorem} 
\newtheorem{definition}[theorem]{Definition}

\newtheorem{proposition}[theorem]{Proposition}

\theoremstyle{remark} 
\newtheorem{examplex}[theorem]{Example}
\newenvironment{example}
  {\pushQED{\qed}\examplex}
  {\popQED\endexamplex}
  \newtheorem{remark}[theorem]{Remark}

\title{Bar Code and Janet-like division}

\author{Michela Ceria\\
 Department of Computer Science - University of Milan - Via Celoria 18, 20133 Milano, Italy \\
               michela.ceria@gmail.com }

\date{ }

\begin{document}

\maketitle

\begin{abstract}
 Bar Codes are combinatorial objects encoding
 many properties of monomial ideals.
 \\
 In this paper we employ these objects to study Janet-like divisions. Given a finite 
 set of terms $U$, from its Bar Code we can 
 compute the Janet-like nonmultiplicative 
 power of its elements and detect completeness of the set. Some observation on the computation of Janet-like bases conclude the work.
\end{abstract}
\section{Introduction}\label{IntroSec}
Bar Codes are combinatorial objects encoding many properties of monomial ideals.
\\
In \cite{Ce}, they have been employed to count zerodimensional (strongly) stable  monomial ideals in 2 and 3 variables
with affine Hilbert polynomial $p \in \NN$, setting a bijection between such ideals and some particular partition of the integer $p$ and then counting these partitions using determinantal formulas.\\
In \cite{CM}, instead, they are the main tool to compute the Groebner escalier of zerodimensional radical ideals given their variety, without passing through the (usually inefficient) Groebner bases' computation.
\\
In this paper, we show that Bar Codes can be successfully used as tools to study, describe and build Janet-like division, i.e. a
divisibility relation on terms, introduced by Gerdt and Blinkov \cite{GB4,GB5} to efficiently compute Groebner bases.\\
Janet-like division, though not being an involutive division \cite{GB1,GB2,GB3}, is strictly related to this concept, being 
a generalization of Janet division \cite{J1} and preserving most of its properties. As Janet division was based on 
the concept of multiplicative/nonmultiplicative variables of the elements of a finite set of terms (the leading terms of a generating set of an ideal, with respect to some term ordering), Janet-like division is based on the concept of nonmultiplicative power for the same terms.
In the case of Janet division, a term $t$ was reducible by a generating polynomial $f$ if and only if $t=\cT(f)w$, where 
$\cT(f)$ was the leading term of $f$ and $w$ a product of powers of multiplicative variables of $\cT(f)$. The case of Janet-like division is analogous, but $w$ should be non-divisible by any nonmultiplicative power of $\cT(f)$.\\
We see in this paper that thanks to Bar Codes it is possible to detect nonmultiplicative powers in a very simple way and to understand if the given generating set is complete, i.e., roughly speaking, if given any $t$ there exists a generator reducing it. If it does not happen, it is possible to update the generating set.\\
Note that the classical cases of Janet/Pommaret division can be easily treated analogously. Other applications of BarCode to Janet decomposition are discussed in \cite{Greedy}.
\\
More precisely, after setting the notation (Section \ref{NotatSec}) and 
giving a brief recap on Bar Codes (Section \ref{BCSec}), we   study Janet-like divisions by means of Bar Codes in Section \ref{JlikeSec} and we  relate Janet nonmultipicative powers to the concept of  {\em infinite corner} (Section \ref{CornerInfSec}).
In the last section, we give an overview on the potential future work on this topic.

\section{Notation}\label{NotatSec}
Throughout this paper we mainly follow the notation of \cite{SPES}. We denote by $\mathcal{P}:=\mathbf{k}[x_1,...,x_n]$ the ring of polynomials in
$n$ variables with coefficients in the field $\ck$. The \emph{semigroup of terms}, generated by the set $\{x_1,...,x_n\}$ is $$\mathcal{T}:=\{x^{\gamma}:=x_1^{\gamma_1}\cdots
x_n^{\gamma_n} \vert \,\gamma:=(\gamma_1,...,\gamma_n)\in \NN^n \}.$$  If $t=x_1^{\gamma_1}\cdots x_n^{\gamma_n}$, then $\deg(t)=\sum_{i=1}^n
\gamma_i$ is the \emph{degree} of $t$ and, for each $h\in \{1,...,n\}$, 
$\deg_h(t):=\gamma_h$ is the $h$-\emph{degree} of $t$.
 A \emph{semigroup ordering} $<$ on $\mathcal{T}$  is  a total ordering
\st $ t_1<t_2 \Rightarrow st_1<st_2,\, \forall s,t_1,t_2
\in \mathcal{T}.$ For each semigroup ordering $<$ on $\mathcal{T}$,  we can represent a polynomial
$f\in \mathcal{P}$ as a linear combination of terms arranged w.r.t. $<$, with
coefficients in the base field $\mathbf{k}$:
$$f\!=\!\sum_{t \in \mathcal{T}}c(f,t)t\!=\!\sum_{i=1}^s c(f,t_i)t_i:\,
c(f,t_i)\in
\mathbf{k}\setminus \{0\},\, t_i\in \mathcal{T}\!,\, t_1>\!...\!>t_s,$$ with
$\cT(f):=t_1$   the 
\emph{leading term} of $f$, $Lc(f):=c(f,t_1)$ the  \emph{leading
coefficient} 
of $f$ and $tail(f):=f-c(f,\cT(f))\cT(f)$  the 
\emph{tail} of $f$.
\\
A \emph{term ordering} is a semigroup ordering which is also a well ordering or, equivalently,  such that $1$ is lower than every variable.\\
In all paper, we consider the \emph{lexicographical ordering} 
induced
by  $x_1<...<x_n$, i.e:
$$ x_1^{\gamma_1}\cdots x_n^{\gamma_n}<_{Lex} x_1^{\delta_1}\cdots
x_n^{\delta_n} \Leftrightarrow \exists j\, \vert  \,
\gamma_j<\delta_j,\,\gamma_i=\delta_i,\, \forall i>j, $$
which is a term ordering. Since we do not consider any 
term ordering other than Lex, we drop the subscript and denote it by $<$ 
instead of $<_{Lex}$.\\
A subset $J \subseteq \kT$ is a \emph{semigroup ideal} if  $t \in J 
\Rightarrow st \in J,\, \forall s \in \mathcal{T}$; a subset ${\sf N}\subseteq \mathcal{T}$ is an \emph{order ideal} if
$t\in {\sf N} \Rightarrow s \in {\sf N}\, \forall s \vert t$.  
We have that ${\sf N}\subseteq \mathcal{T}$ is an order ideal if and only if 
$\mathcal{T}\setminus {\sf N}=J$ is a semigroup ideal.
\\
Given a semigroup ideal $J\subset\mathcal{T}$  we define ${\sf 
N}(J):=\mathcal{T}\setminus J$. The minimal set of generators ${\sf G}(J)$ of $J$ is called \emph{monomial basis} 
of $J$.
 
\noindent For all subsets $G \subset \mathcal{P}$,  $\cT\{G\}:=\{\cT(g),\, g \in  G\}$ and $\cT(G)$ is the semigroup ideal
of leading terms defined as $\cT(G):=\{t \cT(g),\, t \in \mathcal{T}, g \in G\}$. 
\\
Fixed a term order $<$, for any ideal  $I
\triangleleft \mathcal{P}$ the monomial basis of the semigroup ideal 
$\cT(I)=\cT\{I\}$ is called \emph{monomial basis}  of $I$ and denoted again by $\cG(I)$,
whereas the ideal 
$In(I):=(\cT(I))$ is called \emph{initial ideal} and the order ideal 
$\cN(I):=\kT \setminus \cT(I)$ is called \emph{Groebner escalier} of $I$.

\section{Recap on Bar Codes}\label{BCSec}
In this section, referring to \cite{Ce,BCStr}, we summarize the main definitions and properties about Bar Codes, which will be used in what follows. First of all, we recall the general definition of  Bar Code. 
\begin{definition}\label{BCdef1}
A Bar Code $\cB$ is a picture composed by segments, called \emph{bars}, 
superimposed in horizontal rows, which satisfies conditions $a.,b.$ below.
Denote by 
\begin{itemize}
 \item $\cB_j^{(i)}$ the $j$-th bar (from left to right) of the $i$-th row 
 (from top to bottom), $1\leq i \leq n$, i.e. the \emph{$j$-th $i$-bar};
 \item $\mu(i)$ the number of bars of the $i$-th row
 \item $l_1(\cB_j^{(1)}):=1$, $\forall j \in \{1,2,...,\mu(1)\}$ the $(1-)$\emph{length} of the $1$-bars;
 \item $l_i(\cB_j^{(k)})$, $2\leq k \leq n$, $1 \leq i \leq k-1$, $1\leq j \leq \mu(k)$ the $i$-\emph{length} of $\cB_j^{(k)}$, i.e. the number of $i$-bars lying over $\cB_j^{(k)}$
\end{itemize}
\begin{itemize}
 \item[a.] $\forall i,j$, $1 \leq i \leq n-1$, $1\leq j \leq \mu(i)$, $\exists ! \overline{j}\in \{1,...,\mu(i+1)\}$ s.t. $\cB_{\overline{j}}^{(i+1)}$ lies  under  $\cB_j^{(i)}$ 
 \item[b.] $\forall i_1,\,i_2 \in \{1,...,n\}$, $\sum_{j_1=1}^{\mu(i_1)} l_1(\cB_{j_1}^{(i_1)})= \sum_{j_2=1}^{\mu(i_2)} l_1(\cB_{j_2}^{(i_2)})$; we will then say that  \emph{all the rows have the 
same length}.
\end{itemize}
\end{definition}
\begin{example}\label{BC1}
 An example of Bar Code $\cB$ is
 
 
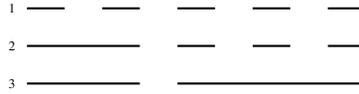
\begin{center}
\begin{tikzpicture}
\node at (3.8,-0.5) [] {${\scriptscriptstyle 1}$};
\node at (3.8,-1) [] {${\scriptscriptstyle 2}$};
\node at (3.8,-1.5) [] {${\scriptscriptstyle 3}$};

\draw [thick] (4,-0.5) --(4.5,-0.5);
\draw [thick] (5,-0.5) --(5.5,-0.5);
\draw [thick] (6,-0.5) --(6.5,-0.5);
\draw [thick] (7,-0.5) --(7.5,-0.5);
\draw [thick] (8,-0.5) --(8.5,-0.5);
\draw [thick] (4,-1)--(5.5,-1);
\draw [thick] (6,-1) --(6.5,-1);
\draw [thick] (7,-1) --(7.5,-1);
\draw [thick] (8,-1) --(8.5,-1);
\draw [thick] (4,-1.5)--(5.5,-1.5);
\draw [thick] (6,-1.5) --(8.5,-1.5);
\end{tikzpicture}
\end{center}
 
 $\quad$ \\
The $1$-bars have length $1$.   As regards the other rows, $l_1(\cB_1^{(2)})=2$,
$l_1(\cB_2^{(2)})=l_1(\cB_3^{(2)})= l_1(\cB_4^{(2)})=1$,
$l_2(\cB_1^{(3)})=1$,$l_1(\cB_1^{(3)})=2$ and
 $l_2(\cB_2^{(3)})=l_1(\cB_2^{(3)})=3$, so
 $\sum_{j_1=1}^{\mu(1)} l_1(\cB_{j_1}^{(1)})= \sum_{j_2=1}^{\mu(2)}
l_1(\cB_{j_2}^{(2)})= \sum_{j_3=1}^{\mu(3)} l_1(\cB_{j_3}^{(3)})=5.$
 \end{example}

\noindent We outline now the construction of the Bar Code associated to a finite set 
of terms. For more details, see \cite{BCStr}, while for an alternative construction, see \cite{Ce}.\\
First of all, given a  term  $t=x_1^{\gamma_1}\cdots 
x_n^{\gamma_n} \in \mathcal{T}\subset \ck[x_1,...,x_n]$, for each $i \in \{1,...,n\}$, we take
$\pi^i(t):=x_i^{\gamma_i}\cdots x_n^{\gamma_n} \in \mathcal{T}.$ 
Taken a finite set of terms $M\subset \mathcal{T}$, for each $ i \in \{1,...,n\}$, we define 
$M^{[i]}:=\pi^i(M):=\{\pi^i(t) \vert t \in M\}.$ We take $M\subseteq \mathcal{T}$, with $\vert M\vert =m < \infty$ and we order its 
elements increasingly  w.r.t. Lex, getting the list  
$\bM=[t_1,...,t_m]$. Then, we construct the sets $M^{[i]}$, and 
the corresponding lexicographically ordered lists\footnote{$\bM$ cannot contain repeated terms, while the $\bM^{[i]}$, for $1<i \leq n$, can. In case some repeated terms occur in $\bM^{[i]}$, $1<i 
\leq n$, they clearly have to be adjacent in the list, due to the 
lexicographical ordering.} $\bM^{[i]}$, for $i=1,...,n$.  We can now define the $n\times m $ matrix of terms $\kM$   s.t. 
its $i$-th row is $\bM^{[i]}$, $i=1,...,n$, i.e.
\[\kM:= \left(\begin{array}{cccc}
\pi^{1}(t_1)&... & \pi^{1}(t_m)\\
\pi^{2}(t_1)&... & \pi^{2}(t_m)\\
\vdots & \quad &\vdots\\
\pi^{n}(t_1)& ... & \pi^{n}(t_m)
\end{array}\right)\]
\begin{definition}\label{BarCodeDiag}
 The \emph{Bar Code diagram} $\cB$ associated to $M$ (or, equivalently, to 
$\bM$) is a 
$n\times m $ diagram, made by segments s.t. the $i$-th row of $\cB$, $1\leq 
i\leq n$,  is constructed as follows:
       \begin{enumerate}
        \item take the $i$-th row of $\kM$, i.e. $\bM^{[i]}$
        \item consider all the sublists of repeated terms, i.e.\\ $[\pi^i(t_{j_1}),\pi^i(t_{j_1 +1}),
        ...,\pi^i(t_{j_1 
+h})]$ s.t. 
        $\pi^i(t_{j_1})= \pi^i(t_{j_1 
+1})=...=\pi^i(t_{j_1 +h})$, noticing that\footnote{Clearly if a term 
$\pi^i(t_{\overline{j}})$ is not 
        repeated in $\bM^{[i]}$, the sublist containing it will be only  
$[\pi_i(t_{\overline{j}})]$, i.e. $h=0$.} $0 \leq h<m$ 
        \item underline each sublist with a segment
        \item delete the terms of $\bM^{[i]}$, leaving only the segments (i.e. 
the \emph{$i$-bars}).
       \end{enumerate}
 We usually label each $1$-bar $\cB_j^{(1)}$, $j \in \{1,...,\mu(1)=m\},$ with the 
term $t_j \in \bM$.
\end{definition}

\noindent A Bar Code diagram is a 
Bar Code in the sense of definition \ref{BCdef1}.

\begin{example}\label{BarCodeNoOrdId}
Given  $M=\{x_1,x_1^2,x_2x_3,x_1x_2^2x_3,x_2^3x_3\}\subset
\mathbf{k}[x_1,x_2,x_3]$, we have the $3 \times 5 $ matrix $\kM$   and   the associated 
Bar Code displayed below:

\[\kM:= \left(\begin{array}{ccccc}
 x_1 & x_1^2 & x_2x_3 & x_1x_2^2x_3 & x_2^3x_3\\
 1 & 1 & x_2x_3 & x_2^2x_3 & x_2^3x_3\\
  1 & 1 & x_3 &  x_3 & x_3
 \end{array}\right)\]
  
\begin{center}
\begin{tikzpicture}
\node at (4.2,0) [] {${\small x_1}$};
\node at (5.2,0) [] {${\small x_1^2}$};
\node at (6.2,0) [] {${\small x_2x_3}$};
\node at (7.2,0) [] {${\small x_1x_2^2x_3}$};
\node at (8.2,0) [] {${\small x_2^3x_3}$};

 \node at (3.8,-0.5) [] {${\scriptscriptstyle 1}$};
\node at (3.8,-1) [] {${\scriptscriptstyle 2}$};
\node at (3.8,-1.5) [] {${\scriptscriptstyle 3}$};

\draw [thick] (4,-0.5) --(4.5,-0.5);
\draw [thick] (5,-0.5) --(5.5,-0.5);
\draw [thick] (6,-0.5) --(6.5,-0.5);
\draw [thick] (7,-0.5) --(7.5,-0.5);
\draw [thick] (8,-0.5) --(8.5,-0.5);
\draw [thick] (4,-1)--(5.5,-1);
\draw [thick] (6,-1) --(6.5,-1);
\draw [thick] (7,-1) --(7.5,-1);
\draw [thick] (8,-1) --(8.5,-1);
\draw [thick] (4,-1.5)--(5.5,-1.5);
\draw [thick] (6,-1.5) --(8.5,-1.5);
\end{tikzpicture}
\end{center}
\end{example} 

It is also possible to associate a finite set of terms $M_\cB$ to a given Bar 
Code $\cB$. In \cite{Ce} we first give a more general procedure to do so and then we specialize it in order 
 to have a \emph{unique} set of terms for each Bar Code.
 
If we apply such specialized procedure to a Bar Code obtained as above from an order ideal $M=\cN$, the unique set we get is exactly $\cN$.

 Here we give only the specialized version, so 
 we follow the steps below:
 \begin{itemize}
 \item[$\mathfrak{B}1$] consider the $n$-th row, composed by the bars 
$B^{(n)}_1,...,B^{(n)}_{\mu(n)}$. Let $l_1(B^{(n)}_j)=\ell^{(n)}_j$, 
for 
$j\in\{1,...,\mu(n)\}$. Label each bar 
$B^{(n)}_j$ with $\ell^{(n)}_j$ copies 
of $x_n^{j-1}$.
 \item[$\mathfrak{B}2$] For each $i=1,...,n-1$, $1 \leq j \leq \mu(n-i+1)$ 
 consider the bar $B^{(n-i+1)}_j$ and suppose that it has been 
 labelled by 
$\ell^{(n-i+1)}_j$ copies of a term $t$. Consider all the $(n-i)$-bars 
$B^{(n-i)}_{\overline{j}},...,B^{(n-i)}_{\overline{j}+h}$ 
  lying immediately  above  $ B^{(n-i+1)}_j$; note that $h$ satisfies 
$0\leq h\leq \mu(n-i)-\overline{j}$. 
 Denote the 1-lengths of 
$B^{(n-i)}_{\overline{j}},...,B^{(n-i)}_{\overline{j}+h}$  by  
$l_1(B^{(n-i)}_{\overline{j}})=\ell^{(n-i)}_{\overline{j}}$,...,
 $l_1(B^{(n-i)}_{\overline{j}+h})=\ell^{(n-i)}_{\overline{j}+h}$. 
 For each $0\leq k\leq h$, label  $ B^{(n-i)}_{\overline{j}+k}$ with 
$\ell^{(n-i)}_{\overline{j}+k}$ copies of $t x_{n-i}^{k}$. 
 \end{itemize}
\begin{definition}\label{Admiss}
A Bar Code $\cB$ is \emph{admissible} if the set $M$ obtained by applying 
$\mathfrak{B}1$ and $\mathfrak{B}2$ to $\cB$  is an order ideal.
\end{definition}

We give now the definition of \emph{block} in a Bar Code and of e-list associated to a $1$-bar, which give a connection between the bars and the terms obtained from the rules 
$\mathfrak{B}1$ and $\mathfrak{B}2$  (see Remark
\ref{ElistExp}).

\begin{definition}\label{Block}
 Given a Bar Code $\cB$, for each $1 \leq l \leq n$, $l\leq i\leq n$, $1\leq j \leq \mu(i)$, 
an $l$-\emph{block} associated to a bar $B_j^{(i)}$ of 
$\sf B$ is the set containing  $B_j^{(i)}$ itself and all the bars
of the 
$(l-1)$ 
rows lying immediately  above $B_j^{(i)}$.
\end{definition}

\begin{definition}\label{elist}
 Given a Bar Code $\cB$, 
 let us consider a $1$-bar $B_{j_1}^{(1)}$, with $j_1 
\in \{1,...,\mu(1)\}$.
 The \emph{e-list} associated to $B_{j_1}^{(1)}$ is the $n$-tuple 
$e(B_{j_1}^{(1)}):=(b_{j_1,n},....,b_{j_1,1})$, defined as follows:
 \begin{itemize}
  \item consider the $n$-bar  $B_{j_n}^{(n)}$, lying under 
  $B_{j_1}^{(1)}$. 
The number of $n$-bars on the left of $B_{j_n}^{(n)}$ is  $b_{j_1,n}.$
  \item for each $i=1,...,n-1$, let  $B_{j_{n-i+1}}^{(n-i+1)}$ and 
$B_{j_{n-i}}^{(n-i)}$ be 
  the $(n-i+1)$-bar and the $(n-i)$-bar 
lying under $B_{j_1}^{(1)}$. Consider the $(n-i+1)$-block associated to 
$B_{j_{n-i+1}}^{(n-i+1)}$, i.e. $B_{j_{n-i+1}}^{(n-i+1)}$ and all the bars lying over it. 
The number of $(n-i)$-bars of 
the block, which lie on  the 
left of $B_{j_{n-i}}^{(n-i)}$ is $b_{j_1,n-i}.$
    \end{itemize}
\end{definition}

\begin{remark}\label{ElistExp}
 Given a Bar Code $\cB$, fix a $1$-bar  $B_{j}^{(1)}$, with $j \in 
\{1,...,\mu(1)\}$.  Comparing definition \ref{elist} and the steps $\mathfrak{B}1$ and 
 $\mathfrak{B}2$ described above, we can observe that the values of the e-list 
$e(B_j^{(1)}):=(b_{j,n},....,b_{j,1})$ are exactly the
exponents of the term 
labelling $B_{j}^{(1)}$, obtained applying $\mathfrak{B}1$ and $\mathfrak{B}2$ to
$\cB$.
\end{remark}

\begin{example}\label{OrdIdElist}
For  the Bar Code $\cB$ 
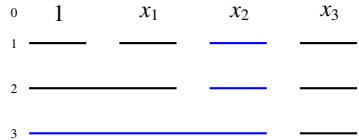
 \begin{center}
\begin{tikzpicture}[scale=0.4]
\node at (-0.5,4) [] {${\scriptscriptstyle 0}$};
\node at (-0.5,0) [] {${\scriptscriptstyle 3}$};
\node at (-0.5,1.5) [] {${\scriptscriptstyle 2}$};
\node at (-0.5,3) [] {${\scriptscriptstyle 1}$};
 \draw [thick, color=blue] (0,0) -- (7.9,0);
 \draw [thick] (9,0) -- (10.9,0);
 
 \draw [thick] (0,1.5) -- (4.9,1.5);
 \draw [thick, color=blue] (6,1.5) -- (7.9,1.5);
 
 \draw [thick] (9,1.5) -- (10.9,1.5);
 
 \draw [thick] (0,3.0) -- (1.9,3.0);
 \draw [thick] (3,3.0) -- (4.9,3.0);
 
 \draw [thick, color=blue] (6,3.0) -- (7.9,3.0);
 
 \draw [thick] (9,3.0) -- (10.9,3.0);
 
 \node at (1,4.0) [] {\small $1$};
 \node at (4,4.0) [] {\small $x_1$};
 \node at (7,4.0) [] {\small $x_2$};
 \node at (10,4.0) [] {\small $x_3$};
\end{tikzpicture}
\end{center}
the e-list of  $B_{3}^{(1)}$ is
$e(B_{3}^{(1)}):=(0,1,0)$;  the bars involved in its computation as stated in Definition \ref{elist} are those highlighted in blue in the above picture.
\end{example}

\begin{proposition}[Admissibility criterion, 
\cite{Ce}]\label{AdmCrit}
 A Bar Code $\cB$ is admissible if and only if, for each 
 $1$-bar $\cB_{j}^{(1)}$, $j \in \{1,...,\mu(1)\}$, the e-list 
$e(\cB_j^{(1)})=(b_{j,n},....,b_{j,1})$ satisfies the following condition: 
$\forall k \in \{1,...,n\} \textrm{ s.t. } b_{j,k}>0,\, \exists \overline{j} 
 \in \{1,...,\mu(1)\}\setminus \{j\} \textrm{ s.t. } $ $$
e(\cB_{\overline{j}}^{(1)})= (b_{j,n},...,b_{j,k+1}, (b_{j,k})-1, 
b_{j,k-1},...,b_{j,1}). $$
\qed
\end{proposition}
\noindent Consider the  sets
$\kA_n:=\{\cB \in \mathcal{B}_n \textrm{ s.t. } \cB  \textrm{ admissible}\} $ and 
$\kN_n:=\{\cN \subset \kT,\, \vert \cN\vert < \infty \textrm{ s.t. } \cN 
\textrm{ is an order ideal}\}.$ We can define the map 
$$\eta: \kA_n \rightarrow \kN_n; \;\;  \cB \mapsto \cN,$$
where $\cN$ is the order ideal obtained applying $\mathfrak{B}1$ and $\mathfrak{B}2$ to $\cB$,
and it can be easily proved that $\eta$ is a bijection.\\

\noindent Up to this point, we have discussed the link between Bar Codes and order ideals,
 i.e. we focused on the link between Bar Codes and Groebner escaliers of 
monomial ideals. We show now that, given an admissible Bar Code 
$\cB$ and the order ideal  $\cN =\eta(\cB)$
it is possible to deduce a very specific generating set 
for the monomial ideal $I$ s.t. $\cN(I)=\cN$.

    \begin{definition}\label{StarSet}
 The \emph{star set} of an order ideal $\cN$ and 
 of its associated Bar Code $\cB=\eta^{-1}(\cN)$ is a set $\kF_\cN$ constructed as 
follows:
 \begin{itemize}
  \item[a)] $\forall 1 \leq i\leq n$, let $t_i$ be a term 
  which labels a $1$-bar lying over $\cB^{(i)}_{\mu(i)}$, 
  then 
  $x_i\pi^i(t_i)\in \kF_\cN$;
  \item[b)] $\forall 1 \leq i\leq n-1$, 
  $\forall 1 \leq j \leq \mu(i)-1$ let 
  $\cB^{(i)}_j$ and $\cB^{(i)}_{j+1}$ be two 
  consecutive bars not lying over the 
same $(i+1)$-bar and let $t^{(i)}_j$ be a term
which labels a $1$-bar lying 
over   $\cB^{(i)}_j$, then 
  $x_i\pi^i(t^{(i)}_j)\in \kF_\cN$.
 \end{itemize}
\end{definition}
\noindent We usually represent $\kF_\cN$ within 
the associated Bar Code $\cB$, inserting
each $t \in \kF_\cN$ on the right of the bar from which 
it is deduced.
Reading the terms from left to right and from the top to
the bottom, $\kF_\cN$ 
is ordered w.r.t. Lex. 

\begin{example}\label{BCP}
For ${\sf
N}=\{1,x_1,x_2,x_3\}\subset
\mathbf{k}[x_1,x_2,x_3]$,  
we have $\kF_\cN=\{x_1^2,x_1x_2,x_2^2,x_1x_3,x_2x_3,x_3^2\}$; looking at 
Definition \ref{StarSet}, we can see that  the terms $x_1x_3,x_2x_3,x_3^2$ come 
from a), while the terms  
$x_1^2,x_1x_2,x_2^2$ come from b).
 
  \begin{center}
\begin{tikzpicture}[scale=0.4]
\node at (-0.5,4) [] {${\scriptscriptstyle 0}$};
\node at (-0.5,0) [] {${\scriptscriptstyle 3}$};
\node at (-0.5,1.5) [] {${\scriptscriptstyle 2}$};
\node at (-0.5,3) [] {${\scriptscriptstyle 1}$};
 \draw [thick] (0,0) -- (7.9,0);
 \draw [thick] (9,0) -- (10.9,0);
 \node at (11.5,0) [] {${\scriptscriptstyle
x_3^2}$};
 \draw [thick] (0,1.5) -- (4.9,1.5);
 \draw [thick] (6,1.5) -- (7.9,1.5);
 \node at (8.5,1.5) [] {${\scriptscriptstyle
x_2^2}$};
 \draw [thick] (9,1.5) -- (10.9,1.5);
 \node at (11.5,1.5) [] {${\scriptscriptstyle
x_2x_3}$};
 \draw [thick] (0,3.0) -- (1.9,3.0);
 \draw [thick] (3,3.0) -- (4.9,3.0);
 \node at (5.5,3.0) [] {${\scriptscriptstyle
x_1^2}$};
 \draw [thick] (6,3.0) -- (7.9,3.0);
 \node at (8.5,3.0) [] {${\scriptscriptstyle
x_1x_2}$};
 \draw [thick] (9,3.0) -- (10.9,3.0);
 \node at (11.5,3.0) [] {${\scriptscriptstyle
x_1x_3}$};
 \node at (1,4.0) [] {\small $1$};
 \node at (4,4.0) [] {\small $x_1$};
 \node at (7,4.0) [] {\small $x_2$};
 \node at (10,4.0) [] {\small $x_3$};
\end{tikzpicture}
\end{center}
\end{example}

In \cite{CMR}, given a monomial ideal $I$, the authors define
 the following set, calling it \emph{star set}:
$$\mathcal{F}(I)=\left\{x^{\gamma} \in \mathcal{T}\setminus {\sf N}(I) \,
\left\vert \,
\frac{x^{\gamma}}{\min(x^{\gamma})} \right. \in {\sf N}(I) \right\}.$$
\begin{proposition}[\cite{Ce}]\label{DefSt}
With the above notation $\mathcal{F}_{\sf N}=\mathcal{F}(I)$.
\end{proposition}
The star set $\kF(I)$ of a monomial ideal $I$ is strongly connected to Janet's 
theory \cite{J1,J2,J3,J4} and to the notion of Pommaret basis \cite{Pom,PomAk,SeiB,SdeltaR}, as explicitly pointed out in \cite{CMR}.
\\In particular, for quasi-stable ideals, the star set is finite and coincides with their Pommaret basis.
 
\section{Bar Code and Janet-like divisions}\label{JlikeSec}
Janet division dates back to the 1920 paper by Janet \cite{J1} and it is first developed  to study partial differential equations via algebraic methods, following and formalizing the approach by Riquier \cite{Riq}.\\
This division is defined, for each set of terms 
$U\subset \mathcal{T}$, as a divisibility relation on terms. In particular, each $t \in U$ is equipped with
a set $M_J(t,U)$ of multiplicative variables, according to the following definition.
   \begin{definition}\cite[ppg.75-9]{J1}\label{multiplicative}
Let  $U\subset \mathcal{T}$ be a set of terms
 and  $t=x_1^{\alpha_1}\cdots x_n^{\alpha_n} $
be an element of $U$.
A variable $x_j$ is called \emph{multiplicative}
for $t$ with respect to $U$ if there is no term in
$U$ of the form
$t'=x_1^{\beta_1}\cdots x_j^{\beta_j}x_{j+1}^{\alpha_{j+1}} \cdots 
x_n^{\alpha_n}$
with $\beta_j>\alpha_j$.
 We denote by $M_J(t,U)$ the set of
multiplicative variables for $t$ with respect to $U$.\\
The variables that are not multiplicative for $t$ w.r.t. $U$ 
are called \emph{non-multiplicative} and we denote by $NM_J(t,U)$ the set 
containing them.
\end{definition}
The divisibility relation is defined as follows: for 
each $u \in \kT$, we say that a term $t\in U$ Janet-divides $u$ if $u=tv$ and each $x_j \mid v$, $j \in \{1,...,n\}$,  belongs 
to $M_J(t,U)$, i.e. $v$ is a product of powers of multiplicative variables for $t$.
In this case, $t$ is a Janet-divisor of $u$ and $u$ a 
Janet-multiple of $t$.
With the definition below, we group together all the 
Janet-multiples  of any term $t \in U$.
\begin{definition}\label{cone}
With the previous notation, the \emph{cone} of  
$t$ with respect to $U$   is the set
$$C_J(t, U):=\{t x_1^{\lambda_1} \cdots x_n^{\lambda_n} \,\vert
\, \textrm{where } \lambda_j\neq 0 \textrm{ only if } x_j \in M_J(t,U)
\}.$$
\end{definition}

It can be proved \cite{J1} that each $u \in \kT$ has at most a 
 Janet-divisor, i.e. the cones are disjoint. A priori, it may happen that a term $u \in \kT$ has no Janet-divisor; the notion of \emph{completeness} 
 characterizes the case in which this cannot happen.
\begin{definition}\label{COMPLETENESS}
  A set $U \subset \kT$ is  \emph{complete} if
  $\cT(U)=\bigcup_{t \in U} C_J(t, U).$  
\end{definition}
Janet division is employed to construct a special kind of Groebner basis for an ideal $I=(G)$ called \emph{Janet basis}. 
Roughly speaking, the complete set $U$ is the set $\cT\{G\}$ of all leading terms for the generators 
and any term $u \in \kT$ is reduced by means of 
the polynomial $f \in G$ such that $t:=\cT(f)\in U$
is the Janet-divisor of $u$. 
\smallskip
 Gerdt and Blinkov \cite{GB1,GB2,GB3} give a generalization of Janet division and Janet bases, 
 by defining \emph{involutive divisions} and 
  \emph{involutive bases} \cite{A,AFS}.\\
In   \cite{GB4,GB5} they introduce 
Janet-like division and Janet-like bases, with the aim to decrease the number of elements in the  basis.\\
We recall now the definitions of  \emph{non-multiplicative power} and of \emph{Janet-like divisor}
from \cite{GB4,GB5}.

\begin{definition}\label{JlikeDivision}
Let $U \subset \kT$ be a finite set of terms; for each $u \in U$, $1\leq i \leq n$  consider 
$h_i(u,U)=\max\{\deg_i(v):\, v \in U, \deg_j(v)=\deg_j(u),\, i+1\leq j \leq n\}-\deg_i(u) \in \NN.$
If $h_i(u,U)>0$, define
$k_i:=\min\{\deg_i(v)-\deg_i(u) :\,  \deg_j(v)=\deg_j(u),\, i+1\leq j \leq n, \deg_i(v)>\deg_i(u)\};$
then $x_i^{k_i}$ is called \emph{non-multipicative power} of $u \in U$.
We denote by $NMP(u,U)$ the set of nonmultiplicative powers for $u \in U$.
\end{definition}

\begin{definition}\label{JlikeDivisors}
Let $U \subset \kT$ be a finite set of terms and $u \in U$; the elements in the monoid ideal 

$$NM(u,U)=\{v \in \kT \vert\, \exists w \in NMP(u,U):\, w \mid v\} $$
are called Janet-like \emph{nonmultipliers} for $u$, whereas the elements in
$M(u,U)=\kT\setminus NM(u,U)$ are called Janet-like \emph{multipliers} for $u$.
\\
A term $u \in U$ is a \emph{Janet-like divisor} of $w \in \kT$ if $w=uv$ with $v \in M(u,U)$.
\end{definition}
\begin{example}
Let us consider the set $U=\{x_1^5,x_2x_1^2,x_2^4x_1,x_3^2x_1^2,$ $ x_3^2x_2^2x_1,x_3^5\}\subset \kT$
of \cite{GB3} and suppose $x_1<x_2<x_3$.
The nonmultiplicative powers are summarized in the following table:
\begin{center}\begin{table}[h]\caption{Nonmultiplicative powers for the terms in $U$.}\label{Table1}
 \begin{tabular}{|c|c|}
\hline
$t$ & $NMP(t,U)$\\
\hline
$x_1^5$& $x_2,x_3^2$\\
\hline
$x_2x_1^2$& $x_2^3,x_3^2$\\
\hline
$x_2^4x_1$ &$x_3^2$\\
\hline
$x_3^2x_1^2$ & $x_2^2,x_3^3$\\
\hline
$x_3^2x_2^2x_1$ & $x_3^3$\\
\hline
$x_3^5$ & $\emptyset$\\
\hline
\end{tabular}  \end{table}
\end{center}
\end{example}
\vspace{-1cm}
We remark that, though Janet-like divisions preserves many properties of Janet division, it is \emph{not an involutive division}.
\\
\smallskip

In what follows, we see that a Bar Code can be used as a tool for studying Janet and Janet-like division.
The construction of a Bar Code can help to assign to each element $t$ of a finite set of terms
$U\subset \kT$ its multiplicative variables, according to Janet's Definition \ref{multiplicative}.\\
Let $U \subset \kT \subset \ck[x_1,...,x_n] $ be a finite set of terms and suppose
$x_1<x_2<...<x_n$.
As explained in section \ref{BCSec}, we can associate a Bar Code $\cB$ to it. Once $\cB$
is constructed, even if  $\cB$  may be a non-admissible Bar Code, we can mimick on it the set up we generally perform to construct the star set\footnote{We put a star symbol $*$ in the diagram in the places where in the star set construction we would have placed the terms.}. In particular:
\begin{itemize}
  \item[a)] $\forall 1 \leq i\leq n$,  place a star symbol $*$ on the right\footnote{The stars we are placing now are in the position corresponding to the terms we find in a) of Definition \ref{StarSet}.} of
  $\cB^{(i)}_{\mu(i)}$;
  \item[b)] $\forall 1 \leq i\leq n-1$, 
  $\forall 1 \leq j \leq \mu(i)-1$ let 
  $\cB^{(i)}_j$ and $\cB^{(i)}_{j+1}$ be two 
  consecutive bars not lying over the 
same $(i+1)$-bar; place a star symbol $*$ between them\footnote{the stars we are placing now are in the position corresponding to the terms we find in b) of Definition \ref{StarSet}.}.
 \end{itemize}
 Now, we state the following proposition, which connects the stars placed above 
 with Janet multiplicative variables.
\begin{proposition}\cite[Prop. 19]{BCVJT}\label{VMolt}
Let $U \subseteq \kT$ be a finite set of terms and let us denote by $\cB_U$ its Bar Code. For each $t \in U$,
$x_i$, $1 \leq i \leq n$, is multiplicative for $t$ if and only if,
in  $\cB_U$, the $i$-bar $\cB^{(i)}_j$, over which $t$ lies, is followed by a star.
\end{proposition}

Now, we start focusing on how to study Janet-like division using Bar Codes. \\
As remarked in \cite{GB4}, \emph{every nonmultiplicative power is nothing else then the power of Janet-nonmultiplicative variable.}
\\
Indeed, consider $u \in U$ and $1\leq i \leq n$. Since $u \in \{ v \in U:\, \deg_j(v)=\deg_j(u),\, i+1\leq j \leq n\}$ we can immediately desume that $h_i(u,U)\geq 0$. In the case $h_i(u,U)=0$, there is no term $ v$ with $\deg_j(v)=\deg_j(u),\, i+1\leq j \leq n$ s.t. $ \deg_i(v)>\deg_i(u)$, so, by Definition \ref{multiplicative},
$x_i$ is Janet-multiplicative for $u$. Otherwise, i.e. if 
$h_i(u,U)>0$, then there is a term $ v$ with $\deg_j(v)=\deg_j(u),\, i+1\leq j \leq n$ s.t. $ \deg_i(v)>\deg_i(u)$, so, again 
by Definition \ref{multiplicative},
$x_i$ is Janet-nonmultiplicative for $u$. This reflects on the Bar Code associated to $U$, since trivially the absence of stars after some bar is equivalent to the presence of a non-multiplicative power of the corresponding variable for the terms over that bar.
Moreover, Janet divisibility implies Janet-like divisibility, whereas the viceversa does not hold (see \cite{GB4} for a proof of this fact).
\\

\smallskip

We prove now the analogous of Proposition \ref{VMolt} for Janet-like division.
\begin{proposition}\label{VJLMolt}
Let $U \subseteq \kT$ be a finite set of terms and let us denote by $\cB_U$ its Bar Code. Let $t \in U$, $x_i \in NM_J(t,U)$ a Janet-nonmultiplicative variable, $\cB^{(i)}_l$ the $i$-bar under $t$ and
$t'$ any term over $\cB^{(i)}_{l+1}$. Then 
$$k_i=\deg_i(t')-\deg_i(t).$$
\end{proposition}
\begin{proof}
We first remark that since  $x_i \in NM_J(t,U)$, by  Proposition \ref{VMolt}, $\cB^{(i)}_l$  is not followed by a star, so again we will find nonmultiplicative powers only when there are no stars.
\\
Now since  $x_i \in NM_J(t,U)$, there is a term $v\in U$ such that 
$\deg_j(v)=\deg_j(t)$, $i+1\leq j\leq n$
 and   $\deg_i(v)>\deg_i(t)$.
In order to find the value $k_i$, we should find the minimal exponent of a term with the same $j$-degree as $t$, $i+1\leq j\leq n$, and bigger $i$-degree.\\
All terms over $\cB^{(i)}_l$ have the same  $\iota-$degree as $t$, $i \leq \iota \leq n$; considering   $\cB^{(i)}_{l+1}$ we have terms which have the same    $\iota-$degree as $t$, $i+1 \leq \iota \leq n$ (if $\cB^{(i)}_{l+1}$ would not be over the same $(i+1)$-bar as $\cB^{(i)}_l$ we would have a star after $\cB^{(i)}_l$). Moreover, their $i$-degree is bigger than $\deg_i(t)$ and it is the minimum with this property due to the Lex ordering of the terms in the Bar Code.
\end{proof}
The concept of \emph{completeness} w.r.t. Janet-like division is analogous to that defined for Janet division in Definition \ref{COMPLETENESS}.
\begin{definition}\label{JLcomplete}
A set $U \subset \kT$ is called \emph{complete} w.r.t. Janet-like division if for the sets
$$C_J(U):=\{uv:\, u \in U, v \in M(u,U)\}$$
and 
$$C(U):=\{uv:\, u \in U, v \in \kT\}$$
holds $C(U)=C_J(U).$
\end{definition}

\begin{proposition}\label{JLContinue}
  A set $U \subset \kT$ is  \emph{complete} w.r.t. Janet-like division if 
  and only if and only if 
  $$\forall u \in U,\, \forall p \in NMP(u,U),\, \exists v \in U :\, v\mid up \textrm{ w.r.t. Janet-like division.} $$
\end{proposition}
 Bar Codes can help us to detect completeness of a finite set of terms, as it is shown in the theorem below.
\begin{theorem}\label{JLCompleteBC}
  Let $U\subset \kT$ be a finite set of terms,
  $\cB$ its Bar Code, $t\in U$, $p=x_i^{k_i}\in NMP(t,U)$ a nonmultiplicative power and $\cB^{(i)}_j$ the $i$-bar under $t$. Let $s \in U$; $s\mid tp$ w.r.t.   Janet-like division if and only if the following conditions hold:
\begin{enumerate}
   \item $s \mid pt$
   \item $s$ lies over  $\cB^{(i)}_{j+1}$ and 
   \item $\forall j'$ such that $x_{j'}\mid \frac{pt}{s}$ either there is a star after the 
$j'$-bar under $s$ or the nonmultiplicative power w.r.t. $x_{j'}$ has greater degree $deg_{j’}(\frac{pt}{s})$.
  \end{enumerate}
\end{theorem}
\begin{proof}
  ``$\Leftarrow$'' It is an obvious consequence of proposition \ref{VJLMolt}; indeed, by 1. $s \mid pt$.  Thanks to (3),
   $\frac{pt}{s}$ is not divided by nonmultiplicative powers of any variable. Notice that $x_i\nmid w:=\frac{pt}{s}$, since  $s$ lies over  $\cB^{(i)}_{j+1}$,
so
 $\deg_i(s)= k_i+\deg_i(t)$ by the minimality of the nonmultiplicative power.
 
 So $sw=pt$ and $w$  does not contain nonmultiplicative powers  for $s$; therefore $s \mid  pt$ w.r.t. Janet-like division.
  \\
  ``$\Rightarrow$''  Let $s \in U$,  $s \mid pt$ w.r.t. Janet-like division; $s\mid pt$ by definition of Janet-like division.\\
  If $s$ would lie over  $\cB^{(i)}_j$, then $\deg_l(s)=\deg_l(t)$ for $l=i,...,n$, i.e. in $s$ and $t$ the variables $x_i,...,x_n$ appear with the same exponent. Then, being $s \mid pt$ and $\deg_i(s)=\deg_i(t)$ , $x_i^{k_i}\mid w:=\frac{pt}{s}$, so either  
$x_i$ is multiplicative for $s$, or the nonmultiplicative 
power of $x_i$ for $s$ is greater than $k_i$.
Both these alternative are impossible: if $x_i$ was multiplicative for $s$ then there would be a star after $\cB^{(i)}_j$, which is impossible by hypothesis, since $p=x_i^{k_i}$ is a nonmultiplicative power for $t$ and they lay over the same $i$-bar. It is also impossible that  the nonmultiplicative 
power of $x_i$ for $s$ is greater than $k_i$ since $\deg_l(s)=\deg_l(t)$ for $l=i,...,n$, and by the minimality of the nonmultiplicative power.
\\
  If $s$ would lie over  $\cB^{(i)}_l$, $l >j+1$, there exists $h\in \{i,...,n\}$ s.t. $\deg_h(s)>\deg_h(pt)$ (remember that the nonmultiplicative power is minimal), so $s \nmid pt$, which is again a contradiction.\\
  If $s$ would lie over  $\cB^{(i)}_l$, $l <j$, then $s<_{Lex} t$ and it cannot happen that $\deg_{l'}(s)=\deg_{l'}(t)$ for $l'=i,...,n$ (since otherwise $s$ would have been over $\cB^{(i)}_j$). Let 
$x_k:=\max\{x_h,\, h=1,...,n \vert \deg_h(s)<\deg_h(t)\}$; it is clear that $k \geq i$. Since $t \in U$ and $\deg_n(t)=\deg_n(s),...,$ $\deg_{k+1}(t)=\deg_{k+1}(s)$ and $\deg_k(t)>\deg_k(s)$,  $x_k$ cannot be a multiplicative variable for $s$. Now, let $x_k^{h_k}$ the nonmultiplicative power of $s$ w.r.t. the variable $x_k$. Being 
$\deg_n(t)=\deg_n(s),...,\deg_{k+1}(t)=\deg_{k+1}(s)$ and $\deg_k(t)>\deg_k(s)$, $h_k \leq \deg_k(t)-\deg_k(s)$, so 
$\deg_k(sx_k^{h_k})\leq \deg_k(t) \leq \deg_k(tp)$, and this is again a contradiction.
  \\Then $s$ must lie over $\cB^{(i)}_{j+1}$.\\
  For being $s \mid pt$, all the variables appearing with nonzero exponent in $\frac{pt}{s}$ must be multiplicative for $s$ or with exponent of non-multiplicative variables smaller than nonmultiplicative powers and this implies that (3) holds. 
\end{proof}
\begin{example}\label{EsGerdt}
Let us consider the set $U=\{x_1^5,x_2x_1^2,x_2^4x_1,x_3^2x_1^2,$ $ x_3^2x_2^2x_1,x_3^5\}\subset \kT$
of \cite{GB4} and suppose $x_1<x_2<x_3$.
 The associated Bar Code is displayed below:
 \begin{center}
\begin{tikzpicture}
 
\node at (4.7,0) [] {${*}$};
 \node at (5.7,0) [] {${*}$};
 \node at (6.7,0) [] {${*}$};
 \node at (7.7,0) [] {${*}$};
 \node at (8.7,0) [] {${*}$};
  \node at (9.7,0) [] {${*}$};

\node at (6.7,-0.5) [] {${*}$};
  \node at (8.7,-0.5) [] {${*}$};
  \node at (9.7,-0.5) [] {${*}$};
 
\node at (9.7,-1) [] {${*}$};

\node at (4.2,0.5) [] {${\small x_1^5}$};
\node at (5.2,0.5) [] {${\small x_2x_1^2}$};
\node at (6.2,0.5) [] {${\small x_2^4x_1}$};
\node at (7.2,0.5) [] {${\small x_3^2x_1^2}$};
\node at (8.2,0.5) [] {${\small x_3^2x_2^2x_1}$};
\node at (9.2,0.5) [] {${\small x_3^5}$};

\draw [thick] (4,0) --(4.5,0);
\draw [thick] (5,0) --(5.5,0);
\draw [thick] (6,0) --(6.5,0);
\draw [thick] (7,0) --(7.5,0);
\draw [thick] (8,0) --(8.5,0);
\draw [thick] (9,0) --(9.5,0);

\draw [thick] (4,-0.5) --(4.5,-0.5);
\draw [thick] (5,-0.5) --(5.5,-0.5);
\draw [thick] (6,-0.5) --(6.5,-0.5);
\draw [thick] (7,-0.5) --(7.5,-0.5);
\draw [thick] (8,-0.5) --(8.5,-0.5);
\draw [thick] (9,-0.5) --(9.5,-0.5);

\draw [thick] (4,-1) --(6.5,-1);
\draw [thick] (7,-1) --(8.5,-1);
\draw [thick] (9,-1) --(9.5,-1);
 
\node at (3.9,0) [] {${\scriptscriptstyle 1}$};
\node at (3.9,-0.5) [] {${\scriptscriptstyle 2}$};
\node at (3.9,-1) [] {${\scriptscriptstyle 3}$};
\end{tikzpicture}
\end{center}
Let us consider the elements in $U$ and identify their nonmultiplicative powers (in complete accordance with Table \ref{Table1}):
\begin{itemize}
  \item $x_1^5$: $x_1$ is multiplicative, the nonmultiplicative powers are $x_2,x_3^2$ since 
  $deg_2(x_2x_1^2)-deg_2(x_1^5)=1$ and  $deg_3(x_3^2x_1^2)-deg_3(x_1^5)= deg_3(x_3^2x_2^2x_1)-deg_3(x_1^5)=2$;
  \item $x_2x_1^2$:  $x_1$ is multiplicative, the nonmultiplicative powers are $x_2^3,x_3^2$ since 
  $deg_2(x_2^4x_1)-deg_2(x_2x_1^2)=3$ and  $deg_3(x_3^2x_1^2)-deg_3(x_2x_1^2)=  deg_3(x_3^2x_2^2x_1)-deg_3(x_2x_1^2)=2$;
  \item $x_2^4x_1$: $x_1,x_2$ are multiplicative, the nonmultiplicative power is $x_3^2$ since 
    $deg_3(x_3^2x_1^2)-deg_3(x_2^4x_1)=  deg_3(x_3^2x_2^2x_1)-deg_3(x_2^4x_1)=2$;
  \item $x_3^2x_1^2$:$x_1$ is multiplicative, the nonmultiplicative powers are $x_3^3,x_2^2$ since 
  $deg_2(x_3^2x_2^2x_1)-deg_2(x_3^2x_1^2)=2$ and  $deg_3(x_3^5)-deg_3(x_3^2x_1^2)=3$;
    \item $x_3^2x_2^2x_1$: $x_1,x_2$ are multiplicative, the nonmultiplicative power is $x_3^3$ since 
    $deg_3(x_3^5)-deg_3(x_3^2x_2^2x_1)=3$;
  \item $x_3^5$: all variables are multiplicative.
\end{itemize}
Now we show that $U$ is a complete set, by multiplying any of its terms by its nonmultiplicative powers and showing that the conditions of Theorem \ref{JLCompleteBC} hold
\begin{itemize}
  \item $x_1^5$: its nonmultiplicative powers are
  $x_2,x_3^2$, so we consider   $x_1^5x_2$ and $x_1^5x_3^2$:
  \begin{itemize}
  \item $x_1^5x_2=(x_1^2x_2)x_1^3$,  so the Janet-like divisor is $x_1^2x_2$;
  \item $x_1^5x_3^2=(x_3^2x_1^2)x_1^3$, so the Janet-like divisor is $x_3^2x_1^2$.
  \end{itemize}
  \item $x_2x_1^2$:    its nonmultiplicative powers are
  $x_2^3,x_3^2$, so we consider   $x_2^4x_1^2$ and $x_3^2x_2x_1^2$:
  \begin{itemize}
  \item $x_2^4x_1^2=(x_2^4x_1)x_1 $,  so the Janet-like divisor is $  x_2^4x_1 $;
  \item $x_3^2x_2x_1^2=(x_3^2x_1^2)x_2 $, so the Janet-like divisor is $x_3^2x_1^2$ (note that, in this case, $x_2$ is not Janet-multipicative for $x_3^2x_1^2$, but the nonmultiplicative power is $x_2^2$, so $x_2 \in M(x_3^2x_1^2,U)$).
  \end{itemize}
  
  \item $x_2^4x_1$:   its nonmultiplicative power is  $x_3^2$, so we have $x_3^2x_2^4x_1=(x_3^2x_2^2x_1)x_2^4$, thus the Janet-like divisor is  $x_3^2x_2^2x_1$;
  \item $x_3^2x_1^2$: its nonmultiplicative powers are  $x_3^3,x_2^2$, so we consider $x_3^5x_1^2 $ and $x_3^2x_2^2x_1^2$:
    \begin{itemize}
  \item $x_3^5x_1^2= (x_3^5)x_1^2 $,  so the Janet-like divisor is $  x_3^5  $;
  \item $x_3^2x_2^2x_1^2=(x_1x_2^2x_3^2)x_1  $, so the Janet-like divisor is $ x_1x_2^2x_3^2$.
  \end{itemize}
    \item $x_3^2x_2^2x_1$: its nonmultiplicative power is  $x_3^3$ so we have  $x_3^5x_2^2x_1=(x_3^5)x_2^2x_1$  thus the Janet-like divisor is $x_3^5$.    
  \item $x_3^5$: all variables are multiplicative, so there is nothing to prove.
\end{itemize}
Note that, in complete accordance with Theorem 
\ref{JLCompleteBC}, for each $t \in U$, the Janet-like divisor with respect to a nonmultiplicative power $x_i^{k_i}$ lies over the subsequent $i$-bar. 
\end{example}

\section{An historical note}\label{CornerInfSec}
In this section, we set a connection between Janet-like multiplicative power and previous results on decomposition of ideals in irreducible primary components.
\\
The first result in this framework dates back to Macaulay \cite{Mac}, who gave an irreducible primary decomposition of a zerodimensional ideal  within a fixed coordinates' system.\\
Such a result has been generalized by Alonso, Marinari an Mora, who gave the definition of {\em infinite corner} \cite{Oracolo}:
\begin{quote}
 Let $I$ be an ideal of $\ck[x_1,...,x_n]$. If $I$ is zerodimensional, its {\em corner set} is defined as 
   ${\bf C}({\sf I}):=\{\tau\in {\bf N} ({\sf I}) \, : \forall \, 1\leq i \leq n, X_i\tau\in{\bf T}({\sf I}) \}\subset {\bf N}({\sf I})$. 
   In the non 0-dimensional case,  the corner set can be generalized \cite{Mac}considering also elements $\tau = x_1^{\alpha_1}\cdots x_n^{\alpha_n},  \alpha_i\in{\Bbb N}\cup\{\infty\}$ and setting 
$$\omega \mid \tau \iff \beta_i \leq \alpha_i \forall \omega= x_1^{\beta_1}\cdots x_n^{\beta_n}.$$

It is then easy to see that there is a finite set 
$${\bf C}^\infty({\sf I})\subset\{x_1^{\alpha_1}\cdots x_n^{\alpha_n}  \alpha_i\in{\Bbb N}\cup\{\infty\}\}$$
which satisfies
$$\omega\in{\bf N} ({\sf I}) \iff \exists \tau\in{\bf C}^\infty({\sf I}) : \omega\mid\tau.$$
\end{quote}
The ideas in \cite{Oracolo} can be interpreted in the language by Gerdt and Blinkov in the
sense   that nonmultiplicative power arise from infinite corners.
The idea behind this connection is to take a generating set $U=\{t_1,...,t_m\} \subset \kT$ for a monomial ideal $J$ and consider it ordered  {\em decreasing order with respect to Lex}, so $t_1>t_2>...>t>m$. First of all we consider the term $t_1$: all multiples of $t_1$ are in $J$ and all the variables are multiplicative for $t_1$ so we say that its infinite corner is $x^\infty y^\infty$.\\
Taken then $t_2$, we want to consider all the multiples of $t_2$ not divided by $t_1$. The infinite corner of $t_2$ with respect to $t_1$ gives the nonmultiplicative powers of $t_2$. In particular, the nonmultiplicative powers are the finite  exponents of the corresponding variables, while the infinite ones represent the multipicative variables. 
Continuing in this fashion with $t_3,...,t_m$, we get all the nonmultipicative powers.
As a simple example, if $U=\{y^3,xy,x^2\}\subset \ck[x,y]$, we have
that the corner of $xy$ with respect to $y^3$ is $x^\infty y^2$
and the corner of $x^2$ with respect to $\{y^3,xy\}$ is $x^\infty y$, as shown in the following picture
\begin{center}
 
  \begin{tikzpicture}[scale=0.8]
\draw 
(2,0) circle (4pt) node[align=right,  below] (x2) {$\scriptstyle{x^2}$}
(1,1) circle (4pt) node[align=right,  below] (xy) {$\scriptstyle{xy}$}
(1,3) circle (1pt) node[align=right,  below] (xy3) {}
(2,1) circle (1pt) node[align=right,  below] (x2y) {}
(9,9) circle (4pt) node[align=right,  below] {$\scriptstyle{x^\infty y^\infty}$}
(9,2.8) circle (4pt) node[align=right,  below] {$\scriptstyle{x^\infty y^2}$}
(9,0.8) circle (4pt) node[align=right,  below] {$\scriptstyle{x^\infty y}$}
(0,3) circle (4pt) node[align=left,   below] {$ \scriptstyle{y^3\quad}$};
\draw[-] (2,0)--(2,1);
\draw[-] (1,1)--(9,1);
\draw[-] (1,1)--(1,9);
\draw[-] (0,3)--(9,3);

\draw[decorate,decoration={brace,amplitude=5pt,mirror}]
(x2.north east) -- (x2y.north east) node [midway,xshift=0.4cm,right] {height 1};
\draw[decorate,decoration={brace,amplitude=5pt,mirror}]
(xy.north east) -- (xy3.north east) node [midway,xshift=0.4cm,right] {height 2};

\draw[->] (0,0)--(0,9);
\draw[->] (0,0)--(9,0);
\end{tikzpicture}\end{center}

\section{Perspectives: reduced Janet-like bases computation}\label{PerspectivesSec}
In this section we give an overview on how to compute 
the Janet-like reduced basis for a zerodimensional radical ideal
$I:=I(\mathbf{X}) \triangleleft \ck[x_1,...,x_n]$, 
given its (finite) variety $\mathbf{X}=\{P_1,...,P_N\}$, in a Groebner-free fashion, following what stated first in \cite{Mou,Lu} and
 explicitly expressed and 
sponsored in the book \cite[Vol.3,40.12,41.15]{SPES}. This approach aims to avoid the computation of a Groebner basis of a (0-dimensional) ideal 
$I\subset \ck[x_1,...,x_n]$ in favour of  combinatorial algorithms describing instead the structure of the quotient algebra $\ck[x_1,...,x_n]/I$.
\\
In the paper \cite{Lu}, Lundqvist proposes four methods to compute the normal form of a polynomial with respect to $I$, without passing through Groebner bases. In particular, we recall the following 
proposition
\begin{proposition}[\cite{Lu}]\label{LNf}
  Let $\mathbf{X}=\{P_1,...,P_N\}$ be a finite set of points, $I:=I(\mathbf{X})\triangleleft \ck[x_1,...,x_n]$ its ideal of points and $\cN=\{t_1,...,t_N\}\subset \ck[x_1,...,x_n]$ such that  $[\cN]=\{[t_1],...,[t_N]\}$ is a basis for $A:=\ck[x_1,...,x_n]/I$. 
  Then, for each $f \in \ck[x_1,...,x_n]$ we 
 have
 \emph{$$\Nf(f,\cN)=(t_1,...,t_N)(\cN[[\mathbf{X}]]^{-1})^t (f(P_1),...,f(P_N))^t,$$}
  where $\Nf(f,\cN)$ is the normal form of $f$ w.r.t. $\cN$ and $\cN[[\mathbf{X}]]$ is the matrix whose rows are the evaluations of the elements of $\cN$ at all the points.
\end{proposition}
If we want to compute a reduced Janet-like basis for 
$I$ given $\mathbf{X}$, we only need:
\begin{itemize}
  \item the points in $\mathbf{X}$;
  \item a basis $\cN$ for the quotient algebra $A:=\ck[x_1,...,x_n]/I$;
  \item a complete set $U$ of terms w.r.t. Janet-like division, which generates $\cT(I)$
\end{itemize}
so that the basis is the set $B=\{\Nf(t,\cN) :\, t \in U \}$.\\
A very simple basis for $A$ is the lexicographical Groebner escalier $\cN(\mathbf{X})$ of $I$ and it can be computed in a purely combinatorial way, without using Groebner bases (see \cite{CM,CeMu,CeMu2,CeMu3,FRR,Lu}). 
Once one has the escalier, it is a trivial task to find a generating set $U$ for $\cT(I)$.
\\
Finally, one can construct the Bar Code associated to $U$ and use 
Theorem \ref{JLCompleteBC} to update it dinamically 
by adding those terms of the form $tv$, $t\in U$, $v \in NMP(t,U)$ such 
that it has no Janet-like divisors in $U$. 
This way, we can get a completion of $U$ and a simple application of Proposition \ref{LNf} to the elements of the completion gives the desired basis, following the approach of \cite{CMV,CMV2}.

 \bibliographystyle{alpha}    
\bibliography{sample-bibliography.bib}   
\end{document}